\documentclass[a4paper,11pt]{amsart}
\usepackage{amsmath}
\usepackage{amsthm,amsfonts,amssymb,graphics}
\usepackage{pgfplots}
\usepackage[foot]{amsaddr}

\usepackage{hyperref}
\usepackage[normalem]{ulem}
\usepackage{enumerate}
\usepackage{geometry}
\usepackage{dsfont}

\newtheorem{theorem}{Theorem}[section]
\newtheorem{lemma}[theorem]{Lemma}
\newtheorem{proposition}[theorem]{Proposition}

\newtheorem{condition}[theorem]{Condition}
\theoremstyle{remark}
\newtheorem{remark}[theorem]{Remark}

\numberwithin{equation}{section}

\newcommand \id{\mathds 1}

\begin{document}
\title[A second moment bound for critical points of planar Gaussian fields]{A second moment bound for \\ critical points of planar Gaussian fields \\ in shrinking height windows}
\author{Stephen Muirhead}
\email{s.muirhead@qmul.ac.uk}
\address{School of Mathematical Sciences, Queen Mary University of London}
\begin{abstract}
We consider the number of critical points of a stationary planar Gaussian field, restricted to a large domain, whose heights lie in a certain interval. Asymptotics for the mean of this quantity are simple to establish via the Kac-Rice formula, and recently Estrade and Fournier proved a second moment bound that is optimal in the case that the height interval does not depend on the size of the domain. We establish an improved bound in the more delicate case of height windows that are shrinking with the size of the domain.
\end{abstract}
\date{\today}
\thanks{We would like to thank Dmitry Beliaev and Michael McAuley for helpful discussions, especially with respect to the application to the number of level and excursion sets \cite{bmm19} that motivated this work, and Igor Wigman for assisting with references. We would also like to thank an anonymous referee for helpful corrections and for pointing out reference \cite{AL}.}

\maketitle

\section{Introduction}
Let $f$ be a $C^1$-smooth stationary planar Gaussian field, and denote by $\kappa(x) = \textrm{Cov}[f(0),f(x)]$ its covariance kernel. For each $R > 0$ and $a \le b$, let $B_R$ denote the ball of radius $R$ centred at the origin, and let $N_R[a, b]$ denote the number of critical points of $f$ inside $B_R$ whose heights (i.e.\ `critical values') lie in the interval $[a, b]$, i.e.,
\[ N_R[a, b] = \# \{  x \in B_R :  f(x) \in [a, b], \nabla f(x) = 0 \} ,\]
where $\# S$ denotes cardinality of a set $S$. A simple application of the Kac-Rice formula shows that, under mild conditions on $\kappa$, the mean of $N_R[a, b]$ is of order $O(R^2 (b-a))$, and it is not difficult to compute asymptotics for $\mathbb{E}[N_R[a,b]]/R^2$ explicitly (see, e.g., \cite{cammarota2014distribution,cheng2015expected} for special cases). On the other hand, the second moment of $N_R[a, b]$ is a more difficult quantity to control, and indeed its finiteness was only established recently \cite{ef_2016} (see also \cite{AL, eli, EL}), with the finiteness of higher moments remaining an important open question. 

Out of the proof of \cite{ef_2016} one can show that there exists a $c > 0$ such that, for each $R \ge 1$ and $a \le b$, $\mathbb{E}[N_R[a, b]^2] \le c R^4$. This bound is of the correct order when the height window $[a, b]$ is fixed (see also \cite{cammarota2017fluctuations,nic_2017}, in which asymptotics for $\mathbb{E}[N_R[a, b]^2]/R^4$ are computed for $[a, b]$ fixed) but is far from optimal if $b-a \to 0$ as $R \to \infty$. Our aim in this note is to derive bounds on $\mathbb{E}[N_R[a, b]^2]$ that remain optimal in the more delicate regime in which $b-a \to 0$ as $R \to \infty$ (`shrinking height windows'). Such bounds have applications in analysing the variance of geometric functionals of planar Gaussian fields, such as the number of level or excursion sets contained in a large domain \cite{bmm19}.

\smallskip
To state our main result we suppose that the following smoothness, non-degeneracy and decay conditions hold:
 \begin{condition}
\label{c:1}
\
\begin{itemize}
\item The covariance kernel $\kappa$ is of class $C^6$.
\item For each $x \in \mathbb{R}^2 \setminus \{0\}$, the Gaussian vector $( f(0), f(x), \nabla f(0), \nabla f(x) ) $ is non-degenerate.
\item As $|x| \to \infty$, $\max_{|\alpha| \le 2} | \partial^\alpha \kappa(x) | \to 0$.
\end{itemize}
\end{condition}

The first condition implies that $f$ is almost surely $C^2$-smooth, and for all multi-indices $\alpha_1$ and $\alpha_2$ such that $ |\alpha_1|, |\alpha_2| \in \{0,1, 2\}$, $(\partial^{\alpha_1} f(0), \partial^{\alpha_2} f(x))$ is Gaussian with covariance  $\textrm{Cov}[ \partial^{\alpha_1} f(0) , \partial^{\alpha_2} f(x)  ] = (-1)^{|\alpha_1|}\partial^{\alpha_1 + \alpha_2} \kappa(x) $.

\smallskip
We also need an extra condition on the support of the spectral measure $\rho$, defined to satisfy $ \kappa(x) = \int_{\mathbb{R}^2} e^{i \langle x, s \rangle } d\rho(s) $.

\begin{condition}
\label{c:2}
The support of $\rho$ is not contained in the union of two lines.
\end{condition}

Conditions \ref{c:1} and \ref{c:2} are extremely mild, and will be satisfied in most applications. Notably, while these conditions imply that $f(0)$ and $\nabla^2 f(0)$ are non-degenerate, we do \textit{not} insist that $(f(0), \nabla^2 f(0))$ be \textit{jointly} non-degenerate, and so they hold in particular for the `random plane wave' (the case $\kappa(x) = J_0(|x|)$, where $J_0$ is the zeroth Bessel function; see, e.g., \cite{beliaev2017two, Berry, NPR}).
 
\smallskip
Our main result on the number of critical points of $f$ is the following:

\begin{theorem}
\label{t:smb}
Suppose $f$ satisfies Conditions \ref{c:1} and \ref{c:2}. Then there exists a $c > 0$ such that, for all $R \ge 1$ and $a \le b$, 
\[  \mathbb{E}[N_R[a, b]^2]  \le c \,  \min\{ R^4 (b-a)^2 + R^2 (b-a) , \, R^4 \}   .  \]
\end{theorem}

\begin{remark}
The bound exhibits crossover behaviour if $b-a \ll 1/R^2$, which is the regime in which $\mathbb{E}[N_R[a, b]^2] \ll 1$, and hence also $\mathbb{P}[N_R[a, b] \ge 1] \ll 1$. 
\end{remark}

\begin{remark}
As in \cite{AL,ef_2016}, we could probably replace the condition that $\kappa$ is $C^6$ with the weaker condition that $\kappa$ is $C^{4+}$ and satisfies a \textit{Geman condition} \cite{geman,KL}, i.e.\ there exists a $\delta > 0$ such that
\[    \max_{|\alpha| = 4}  \int_{|x| < \delta} \frac{ | \partial^\alpha \kappa(x) - \partial^\alpha \kappa(0) | }{|x|^2} < \infty .\]
Since optimum conditions for Theorem \ref{t:smb} are not our primary interest, we work with the simpler condition here.
\end{remark}

\begin{remark}
It is likely that our analysis could extend to higher dimensional fields (as in~\cite{ef_2016}), but this would increase the computational complexity of our proof (especially of Lemma~\ref{l:2}). On the other hand, our analysis goes through unchanged in the (easier) one-dimensional case; we discuss this in the appendix.
\end{remark}

Naturally, the constant $c$ in Theorem \ref{t:smb} depends on the Gaussian field $f$. Our second result gives a bound that is uniform over a collection of Gaussian fields, which is useful in applications \cite{bmm19}. 

\begin{theorem}
\label{t:smb2}
Let $(f_i)_{i \in \mathcal{I}}$ be a collection of continuous (not necessarily stationary) Gaussian fields, each defined on a compact domain $D_i \subset \mathbb{R}^2$, with a $C^{3,3}$-smooth covariance kernel. Let $N_i[a, b]$ be the number of critical points of $f_i$ in $D_i$ whose heights lie in $[a, b]$. Suppose that:
\begin{enumerate}
\item The fields are normalised so that, for each $i \in \mathcal{I}$ and $x \in D_i$,
\[ \mathbb{E}[f_i(x)]  = 0 \, , \ \textrm{Var}[f_i(x)] = 1 \, , \ \text{Cov}[f_i(x), \nabla f_i(x) ] = 0 \quad \text{and} \quad   \text{Cov}[ \nabla f_i(x), \nabla f_i(x) ] = \id_2 , \]
and $( f_i(x), f_i(y), \nabla f_i(x), \nabla f_i(y) )$ is non-degenerate for all distinct $x, y \in D_i$;
\item There exists a constant $c_1 > 0$ such that $\sup_{i \in \mathcal{I}} \sup_{x \in D_i } \max_{|\alpha| \le 3}   \textrm{Var}[\partial^{\alpha} f_i(x)] < c_1  $;
\item There exists a constant $c_2>0$ such that
\[  \inf_{i \in \mathcal{I}} \inf_{x \in D_i, v \in \mathbb{S}^1}   \textrm{Var}[\partial_v^{(2, 0)} f_i(x)]  > 1 + c_2 \quad \text{and} \quad \inf_{i \in \mathcal{I}} \inf_{x \in D_i, v \in \mathbb{S}^1}  \textrm{Var}[\partial_v^{(1, 1)} f_i(x)]  > c_2  ,   \]
where $\partial_v$ denotes the derivative with respect to coordinate axes in the $v$ direction;  
\item For each $\delta > 0$, there exists a constant $c_3 > 0$ such that
\[ \inf_{i \in \mathcal{I}} \inf_{|x-y|  \ge \delta} \det(\Sigma^i_1(x,y))  > c_3 \quad \text{and} \quad  \inf_{i \in \mathcal{I}}\inf_{ |x-y| \ge \delta} \det(\Sigma^i_2(x,y))  > c_3 , \]
where  $\Sigma^i_1(x,y)$ and $\Sigma^i_2(x,y)$ denote respectively the covariance matrices of 
\[     (f_i(x), f_i(y) ) \, | \, (\nabla f_i(x), \nabla f_i(y) )  \quad \text{and} \quad  (\nabla f_i(x), \nabla f_i(y) )  .\]
\end{enumerate}
Then there exists a $c > 0$ such that, for all $i \in \mathcal{I}$ and $a \le b$,
\[  \mathbb{E}[N_i[a, b]^2]  \le c \,  \min\{ \text{Area}(D_i)^2 (b-a)^2 + \text{Area}(D_i) (b-a) , \, \text{Area}(D_i)^2  + \text{Area}(D_i) \}   .  \]
\end{theorem}


\section{Proof of the second moment bound}

We shall prove Theorem \ref{t:smb} as a corollary of Theorem \ref{t:smb2}. Let $f$ be a continuous Gaussian field on a compact domain $D \subset \mathbb{R}^2$ with a $C^{3,3}$-smooth covariance kernel. Suppose that $f$ is normalised so that
\begin{equation}
\label{e:norm2}
\mathbb{E}[f(x)]  = 0 \, , \ \textrm{Var}[f(x)] = 1 \, , \ \text{Cov}[f(x), \nabla f(x) ] = 0 \quad \text{and} \quad   \text{Cov}[ \nabla f(x), \nabla f(x) ] = \id_2 ,
\end{equation}
and the vector $( f(x), f(y), \nabla f(x), \nabla f(y) )$ is non-degenerate for all distinct $x, y \in D$.

\smallskip
We begin by introducing a parameter $\delta > 0$, and splitting 
\[ N[a, b]^2 =  \# \{ (x,y) \in D \times D :   f(x) \in [a, b],  f(y) \in [a, b] ,\nabla f(x) = \nabla f(y) = 0 \}   \]
into three terms
\begin{align*}
N^1[a, b; \delta] & =  \# \{ (x,y) \in D \times D  :  |x - y| > \delta, (f(x),  f(y)) \in [a, b]^2,\nabla f(x) = \nabla f(y) = 0 \}  \\
N^2[a, b; \delta] & = \# \{ (x,y) \in D \times D  :  0 < |x-y| \le \delta,  (f(x),  f(y)) \in [a, b]^2, \nabla f(x) = \nabla f(y) = 0 \} \\
N^3[a, b] & = \# \{ (x,x) \in D \times D  :   f(x) \in [a, b], \nabla f(x) = 0  \} 
\end{align*}
so that $N[a, b]^2 =   N^1[a, b; \delta] + N^2[a, b; \delta]   +  N^3[a, b] $. A simple application of the Kac-Rice formula yields the following upper bounds on the expectation of each term:

\begin{proposition}
\label{p:kr}
There exists an absolute constant $c > 0$ such that, for each $a \le b$ and $\delta > 0$,
\begin{align*}
& \mathbb{E}[N^1[a, b; \delta]  ]  \le c \, \text{Area}(D)^2 (b-a)^2 \,  \sup_{ |x-y| \ge \delta} \, \sup_{s, t \in [a, b]}  I^1(x,y; s, t)   , \\
& \mathbb{E}[N^2[a, b; \delta]  ]  \le  c \, \text{Area}(D) (b-a) \,    \sup_{0 <|x-y| \le \delta} \,  \sup_{s \in [a, b]} I^2(x,y; s)   ,
\end{align*}
and
\[   \mathbb{E}[N^3[a, b]  ] \le  c \, \text{Area}(D) (b-a)  \,  \sup_{x \in D} \, \sup_{s \in [a, b]}  I^3(x;s) ,  \]
where $I^1$, $I^2$ and $I^3$ denote the intensity functions
\begin{align*}
  I^1(x,y; s, t) & =   \gamma_{x,y}^1(s, t, 0, 0)  \times   \mathbb{E}[ |\textup{det}(\nabla^2 f(x) \nabla^2 f(y) ) |   \, | \,  f(x) = s , f(y) = t , \nabla f(x) = \nabla f(y) = 0 ] , \\
 I^2(x,y; s) & =   \gamma_{x,y}^2(s, 0, 0) \times  \mathbb{E}[ |\textup{det}(\nabla^2 f(x) \nabla^2 f(y) ) |  \, |  \, f(x) = s , \nabla f(x) = \nabla f(y) = 0 ]    , \\
  I^3(x;s) & =  \gamma_x^3(s,0)  \times  \mathbb{E}[ |\textup{det}(\nabla^2 f(x) ) |  \, | \,  f(x) = s ,  \nabla f(x)  = 0 ]  , 
  \end{align*}
and where $\gamma_{x,y}^1$, $\gamma_{x,y}^2$ and $\gamma_x^3$ denote, respectively, the densities of the (non-degenerate) Gaussian vectors
\[     ( f(x),  f(y), \nabla f(x), \nabla f(y) )  \ , \quad  ( f(x), \nabla f(x), \nabla f(y) )  \quad \text{and} \quad ( f(x), \nabla f(x) )  .   \]
Moreover, $I^1$, $I^2$ and $I^3$ are continuous on $(\mathbb{R}^2 \setminus \{(x,x)\} ) \times \mathbb{R}^2$, $(\mathbb{R}^2 \setminus \{(x,x)\} ) \times \mathbb{R}$ and $\mathbb{R}$ respectively.
\end{proposition} 

\begin{proof}
This is a direct application of the Kac-Rice formula \cite[Theorem 6.3]{azais2009level} after bounding the relevant integrands by their suprema; the Kac-Rice formula is valid in our setting since $f$ is almost surely $C^2$ and the vector $( f(x), f(y), \nabla f(x), \nabla f(y) ) $ is non-degenerate for $x \neq y$.
\end{proof}

In the case of large height window ($b - a \gg 1$), we bound $N[a,b]^2$ more simply as follows:
\begin{proposition}
\label{p:kr2}
There exists an absolute constant $c > 0$ such that, for each $a \le b$,
\[ \mathbb{E}[N[a, b]^2]  \le c \big(  \text{Area}(D)^2  \,   \sup_{ |x-y| > 0} I^4(x,y)  +  \text{Area}(D)  \, \sup_{x \in D} I^5(x) \big) ,  \]
where $I^4$ and $I^5$ denote the intensity functions
\begin{align*}
  I^4(x,y) & =   \gamma_{x,y}^4(0, 0)  \times   \mathbb{E}[ |\textup{det}(\nabla^2 f(x) \nabla^2 f(y) ) |   \, | \, \nabla f(x) = \nabla f(y) = 0 ] , \\
  I^5(x) & =  \gamma_x^5(0)  \times  \mathbb{E}[ |\textup{det}(\nabla^2 f(x) ) |  \, | \,   \nabla f(x)  = 0 ]  , 
  \end{align*}
and where $\gamma_{x,y}^4$ and $\gamma_x^5$ denote, respectively, the densities of the (non-degenerate) Gaussian vectors
\[   (  \nabla f(x), \nabla f(y)  )  \quad \text{and} \quad  \nabla f(x)   .   \]
Moreover, $I^4$ and $I^5$ are continuous on $(\mathbb{R}^2 \setminus \{(x,x)\} )$ and $\mathbb{R}$ respectively.
\end{proposition} 
\begin{proof}
This is again an application of the Kac-Rice formula \cite[Theorem 6.3]{azais2009level}.
\end{proof}

The technical heart of the proof is to establish the following bounds on the intensity functions:
\begin{lemma}[Off-diagonal part]
\label{l:1}
Let $\delta > 0$ be given. Suppose that there exist $c_1, c_2> 0$ such that
\begin{equation}
\label{e:l1}
 \sup_{x \in D} \max_{|\alpha| \le 2}   \textrm{Var}[\partial^{\alpha} f(x)]  < c_1 \quad \text{and} \quad  \inf_{|x-y| \ge \delta }  \min_{i \in \{1,2\}} \det(\Sigma_i(x,y))   > c_2   , 
 \end{equation}
where  $\Sigma_1(x,y)$ and $\Sigma_2(x,y)$ denote respectively the covariance matrices of the vectors
\begin{equation}
\label{d:s12}
    (f(x), f(y) ) \, | \, (\nabla f(x), \nabla f(y) )  \quad \text{and} \quad  (\nabla f(x), \nabla f(y) )  .
    \end{equation}
Then there exists a $c > 0$, depending only on $\delta, c_1$ and $c_2$, such that $\sup_{|x-y| \ge \delta}  \sup_{s,t \in \mathbb{R}}  I^1(x,y; s, t)  < c $.
\end{lemma}

\begin{lemma}[Near-diagonal part]
\label{l:2}
Suppose that there exist $c_1, c_2 > 0$ such that 
\begin{equation}
\label{e:l21}
 \sup_{x \in D} \max_{|\alpha| \le 3}  \textrm{Var}[ \partial^{\alpha_1} f_i(x) ] < c_1
\end{equation}
and
\begin{equation}
\label{e:l22}
 \inf_{x \in D, v \in \mathbb{S}^1}    \textrm{Var}[\partial_v^{(2, 0)} f(x) ] > 1+ c_2  \, , \quad \inf_{x \in D, v \in \mathbb{S}^1}     \textrm{Var}[\partial_v^{(1,1)} f(x)]  > c_2  . 
  \end{equation}
Then there exist $\delta, c > 0$, depending only on $c_1$ and $c_2$, such that $ \sup_{0<|x-y| \le \delta}  \sup_{s \in \mathbb{R}} I^2(x,y; s)    < c $.
\end{lemma}

\begin{lemma}[On-diagonal part]
\label{l:3}
Suppose that there exist $c_1 > 0$ such that
\begin{equation}
\label{e:l3}
  \sup_{x \in D}  \max_{|\alpha| \le 2}   \textrm{Var}[\partial^{\alpha} f(x)]  < c_1 . 
  \end{equation}
Then there exists a $c > 0$, depending only on $c_1$, such that $\sup_{x \in D}  \sup_{s \in \mathbb{R}} I^3(x;s)  < c $.
\end{lemma}

\begin{lemma}[No height window]
\label{l:4}
Suppose that there exist $c_1, c_2 > 0$ such that 
\begin{equation}
\label{e:l41}
 \sup_{x \in D} \max_{|\alpha| \le 3}  \textrm{Var}[ \partial^{\alpha_1} f_i(x) ] < c_1 ,
  \end{equation} 
 \[  \inf_{x \in D, v \in \mathbb{S}^1}    \textrm{Var}[\partial_v^{(2, 0)} f(x) ] > 1+ c_2  \quad \text{and} \quad  \inf_{x \in D, v \in \mathbb{S}^1}     \textrm{Var}[\partial_v^{(1,1)} f(x)]  > c_2  .  \]
  Then there exist $\delta, c > 0$, depending only on $c_1$ and $c_2$, such that $\sup_{0<|x-y| \le \delta} I^4(x,y)  < c $ and $\sup_{x\in D} I^5(x) < c$.
Moreover, let $\delta > 0$ be given and suppose there exists $c_3 > 0$ such that
\[ \inf_{|x-y| \ge \delta }  \det(\Sigma_2(x,y))   > c_3   , \]
 where $\Sigma_2(x, y)$ is as in \eqref{d:s12}. Then there exists $c > 0$, depending only on $\delta$ and $c_3$, such that $\sup_{|x-y| \ge \delta} I^4(x,y)  < c $.
\end{lemma}

The proofs of Lemmas \ref{l:1}--\ref{l:4} reduce to some Gaussian computations which we carry out in the next section. Let us conclude this section by showing how they imply Theorems~\ref{t:smb} and \ref{t:smb2}.

\begin{proof}[Proof of Theorem \ref{t:smb2}]
Under the assumptions of Theorem \ref{t:smb2}, the constant $\delta > 0$ appearing in Lemmas \ref{l:2} and~\ref{l:4} can be chosen uniformly for all $(f_i)_{i \in \mathcal{I}}$. Fix such a $\delta > 0$. Then, again under the assumptions of Theorem \ref{t:smb2}, the conditions in Lemmas \ref{l:1}--\ref{l:4} hold uniformly for all $(f_i)_{i \in \mathcal{I}}$. The proof then follows by combining Propositions~\ref{p:kr} and \ref{p:kr2}, and Lemmas \ref{l:1}--\ref{l:4}.
\end{proof}

\begin{proof}[Proof of Theorem \ref{t:smb}]
By stationarity and since $(f(x), \nabla f(x))$ is non-degenerate, via a linear rescaling of $f$ and the domain $\mathbb{R}$ we may assume the normalisation 
\begin{equation}
\label{e:norm}
 \mathbb{E}[f(x)]  = 0 \, , \ \textrm{Var}[f(x)] = 1 \, , \ \text{Cov}[f(x), \nabla f(x) ] = 0 \quad \text{and} \quad   \text{Cov}[ \nabla f(x), \nabla f(x) ] = \id_2 . 
 \end{equation}
This normalisation changes $\sup_{|b-a| = \lambda} N_R[a,b]$ by a multiplicative constant that does not depend on $\lambda$ and $R$, and so does not affect the conclusion of Theorem \ref{t:smb}. 

\smallskip
It suffices to show that, under Conditions \ref{c:1} and \ref{c:2}, the assumptions in Theorem \ref{t:smb2} are satisfied for $f_i = f$, $D_i = B_i$, and $\mathcal{I} =  [1,\infty)$.

\smallskip
\noindent (1)--(2). Immediate from \eqref{e:norm} and the fact that $\kappa$ is $C^6$.

\noindent (3). Fix $v \in \mathbb{S}^2$ and align the coordinate axis with $v$. By stationarity and the Cauchy-Schwarz inequality applied in Fourier space,
\[  \textrm{Var}[ \partial_v^{(2, 0)} f(x)] =  \int_{s =(s_1, s_2)} s_1^4 \, d\rho(s)  \ge \Big(  \int_{s=(s_1, s_2)}  s_1^2 \, d\rho(s) \Big)^2 = \big( \textrm{Var}[\partial_v^{(1, 0)} f(x)] \big)^2 = 1,   \]
with equality if and only if the spectral measure $\rho$ is supported on a pair of parallel lines $\{ |s_1| = k \}$, $k \ge 0$. Similarly, 
\[ \textrm{Var}[\partial_v^{(1, 1)} f(x)] =   \int_{s =(s_1, s_2)} s_1^2 s_2^2  \, d\rho(s) \ge 0, \]
 with equality if and only if the spectral measure $\rho$ is supported on the lines $\{ |s_1| = 0 \} \cup \{ |s_2| = 0 \}$. Since Condition~\ref{c:2} rules out the cases of equality, and since $\mathbb{S}^1$ is compact, we validate the assumption.
 
\noindent (4). Let $\Sigma_1(x)$ and $\Sigma_2(x)$ be the covariance matrices defined in \eqref{d:s12}, and observe that these are strictly positive-definite under Condition \ref{c:1}. By Gaussian regression (\cite[Proposition 1.2]{azais2009level})
\[  \Sigma_1(x) = M_{11} - M_{12} M_{22}^{-1} M_{12}^T \quad \text{and} \quad \Sigma_2(x) = M_{22} , \]
  where
  \[ M_{11}  =  \left[ {\begin{array}{cc}
   1 & \kappa(x)  \\
 \kappa(x) & 1 \\
  \end{array} } \right]   , \ M_{12}  =  \left[ {\begin{array}{cc}
   0 & -\nabla \kappa(x)  \\
\nabla \kappa(x) & 0 \\
  \end{array} } \right]  \quad \text{and} \quad M_{22}  =  - \left[ {\begin{array}{cc}
   -\id_2 & \nabla^2 \kappa(x)  \\
\nabla^2 \kappa(x) & -\id_2\\
  \end{array} } \right]   \]
are also strictly positive-definite. Since both determinants and inverses are continuous with respect to the entry-wise sup-norm  on the set of strictly positive-definite matrices, this implies that $\det(\Sigma_{1}(x))$ and $\det(\Sigma_2(x))$ are strictly positive and continuous in $x$. Since, under Condition~\ref{c:1},  $\lim_{|x| \to \infty} \max_{|\alpha| \le 2} | \partial^\alpha \kappa(x)| = 0$, it follows that 
\[  \lim_{|x| \to \infty} \textup{det}(\Sigma_1(x) )  = 1      \quad \text{and} \quad  \lim_{|x| \to \infty} \textup{det}(\Sigma_2(x))  = \det(\id_2)^2 = 1  \]
the so the assumption is validated by the continuity of $\det(\Sigma_i(x))$ (and the stationarity of $f$).
\end{proof}

\section{Gaussian computations}
To assist in proving Lemmas \ref{l:1}--\ref{l:4}, we rely on the following auxiliary lemma:

\begin{lemma}
\label{l:det}
Fix $d \in \{1, 2\}$ and $n \in \mathbb{N}$. Let $X$ be a random $2 \times 2$ matrix, let $Y \in \mathbb{R}^d$ and $Z \in \mathbb{R}^{4}$ be random vectors, and suppose that $(X, Y, Z)$ is jointly Gaussian and centred, with $(Y, Z)$ non-degenerate. Let $\varphi$ and $\Sigma$ denote respectively the density and covariance matrix of $(Y, Z)$, and let $\Sigma_{Y|Z}$ denote the covariance matrix of $Y \, | \,Z$ (which does not depend on $Z$ by Gaussian regression). Then there exists a constant $c > 0$, depending only on $n$, such that 
\[ \sup_{y \in \mathbb{R}^d} \varphi(y, 0)  \, \mathbb{E}[  |\textup{det}(X)|^n \, |\, Y = y, Z = 0   ] \]
is bounded above by
\begin{equation}
\label{e:ldet1}
 \frac{c}{ \sqrt{ \textup{det}(\Sigma)} } \bigg(   \prod_{\text{largest two}}    \mathbb{E} \big[X^2_{i, j} \, | \, Z = 0\big]^{n/2}   \bigg) \max \bigg\{   1  , \,  \frac{  \max_k \mathbb{E} \big[ Y_k^2 \big]^{2n}  }{ \textup{det}(\Sigma_{Y|Z})^n }  \bigg\}  , 
 \end{equation}
 where $ \prod_{\text{largest two}}(\cdot)$ denotes the product of the largest two entries of a positive $2 \times 2$ matrix. In turn, \eqref{e:ldet1} is bounded above by
 \begin{equation}
\label{e:ldet2}
 \frac{c}{ \sqrt{ \textup{det}(\Sigma)} } \Big(  \max_{i, j}  \mathbb{E} \big[X^2_{i, j} \big]^n   \Big) \max \bigg\{   1  , \,  \frac{  \max_k \mathbb{E} \big[ Y_k^2 \big]^{2n}  }{ \textup{det}(\Sigma_{Y|Z})^n }  \bigg\}  .
 \end{equation}
 \end{lemma}
 
 \begin{remark}
Lemma \ref{l:det} can be compared to \cite[Lemma A.4]{topmixing18}, in which a similar bound was established.
 \end{remark}

\begin{proof} 
Let $c$ denote a positive constant, depending only on $n$, that may change from line to line. Throughout the proof we repeatedly use the fact that conditioning on part of a Gaussian vector reduces the variance of all coordinates. If $M = (M_{i,j})$ is a $2 \times 2$ matrix, then by expanding the determinant it is immediate that
\[ |\textup{det}(M)|^n \le c  \big(  |M_{1,1}^n M_{2,2}^n| + |M_{1,2}^n M_{2,1}^n| \big) .\]
Hence, applying H\"{o}lder's inequality,
\begin{align*}
 \mathbb{E}[ |\textup{det}(X)|^n \, | \, Y = y, Z = 0 ]  &\le    c \Big(   \mathbb{E} \Big[  |X_{1,1}^n X_{2,2}^n| \, | \, Y = y, Z = 0 \Big] +  \mathbb{E} \Big[ | X_{1,2}^n X_{2,1}^n| \, | \, Y = y, Z = 0 \Big] \Big)  \\
 & \le c  \prod_{\text{largest two}}  \Big( \mathbb{E}[X_{i,j}^{2n} \, | \, Y = y, Z = 0] \Big)^{1/2}  .
 \end{align*}
Since a normally distributed random variable $Z_0 \sim \mathcal{N}(\mu, \sigma^2)$ satisfies $\mathbb{E}[ Z_0^{2n} ] \le c \max\{  (\sigma^2)^n, \mu^{2n}  \}$, we have that $\mathbb{E} [ X_{i, j}^{2n}  \, |\, Y = y, Z=0 ]$ is bounded above by
\[ c   \max \Big\{  \mathbb{E} \big[ X^2_{i, j}  \, |\, Y = 0, Z = 0  \big]^n \, , \   \mathbb{E} \big[ X_{i, j}  \, |\, Y = y , Z=0 \big]^{2n} \Big\}   . \]
Recalling that
\[  \varphi(y, 0)   \le c \frac{ e^{-\frac{1}{2} y^T \Sigma^{-1}_{Y|Z} y } }{ \sqrt{ \textup{det}(\Sigma)   }  },  \]
 and since $ \mathbb{E} \big[ X^2_{i, j}  \, |\, Y = 0, Z = 0  \big] \le  \mathbb{E} \big[ X^2_{i, j}  \, |\,  Z = 0  \big] $, to establish \eqref{e:ldet1} it remains to show that
\begin{equation}
\label{e:ldet3}
 \sup_{y \in \mathbb{R}^d} \Big\{ \mathbb{E} \big[ X_{i, j}  \, |\, Y = y , Z = 0 \big]^{2n}  e^{-\frac{1}{2} y^T \Sigma_{Y|Z}^{-1} y }   \Big\} \le c  \,      \frac{   \mathbb{E} \big[X_{i, j}^2 \, | \, Z = 0\big]^n   \max_k  \mathbb{E} \big[ Y_k^2 \big]^{2n} }{ \textup{det}(\Sigma_{Y|Z})^n  }    .
 \end{equation} 
For this, write $\Sigma_Y^{-1} = U^T \Lambda^{-1} U$, where $U = (u_{k_1, k_2})$ is a $d \times d$ orthogonal matrix and $\Lambda = \textup{Diag}(\lambda_k)$ is the $d \times d$ diagonal matrix of (strictly positive) eigenvalues of $\Sigma_{Y|Z}$. Abbreviating $S = (s_k) : = U \mathbb{E}[X_{i, j} Y_k | Z = 0]$ and replacing $y$ by $U y$, by Gaussian regression we have that
\[   \sup_{y \in \mathbb{R}^d}  \Big\{  \mathbb{E} \big[ X_{i, j}  \, |\, Y = y ,Z =0  \big]^{2n}  e^{-\frac{1}{2} y^T \Sigma_{Y|Z}^{-1} y }  \Big\} \le  \sup_{y \in \mathbb{R}^d} \Big\{  \big( S^T \Lambda^{-1} \, y \big)^{2n}  \, e^{ -\frac{1}{2} y^T \Lambda^{-1} y }  \Big\} . \]
Differentiating in $y$ and computing explicitly, the maximum of the expression on the right-hand side is attained, in the case $d = 1$, at
\[ y = \begin{cases}
 \pm  \sqrt{2n \lambda_1}, & s \neq 0, \\
 0, & s = 0 , \\
 \end{cases}
  \]
and, in the case $d = 2$, at
\[ y = (y_1, y_1) = \begin{cases}
 \frac{\pm \sqrt{2n}}{ \sqrt{ s_1^2 \lambda_1^{-1} + s_2^2 \lambda_2^{-1}}    }  \big( s_1, s_2 \big), & (s_1, s_2) \neq (0, 0)  ,  \\
(0, 0) , & (s_1, s_2) = (0,0) . \\
 \end{cases}   \]
In both cases, this yields a maximum value of
\[   (2n/e)^n \, \Big( \sum_k s_k^2  \lambda_k^{-1}  \Big)^n  \le c \, \Big(  \max_k s_k^2  \max_k  \lambda_k^{-1}   \Big)^n    . \]
Since the eigenvalues of a positive-definite real-symmetric matrix are bounded by a constant times the maximum diagonal entry,
\[   \max_k \lambda_k^{-1} =   \frac{ \max_k \lambda_k }{ \textup{det}(\Lambda)} \le c \, \frac{ \max_k \mathbb{E} \big[ Y_k^2 | Z = 0 \big]  }{ \textup{det}(\Sigma_{Y|Z})  }   .  \]
Moreover, since $U$ has entries bounded above in absolute value by $1$ (being orthogonal), and by the Cauchy-Schwarz inequality,
\[  \max_k  s_k  \le   c \, \max_k  |\mathbb{E}[X_{i, j} Y_k \, | \, Z = 0] |  \le c \,  \mathbb{E}[X_{i, j}^2 \, | \, Z = 0]^{1/2}  \,  \max_k \mathbb{E} \big[ Y_k^2 \, | \, Z = 0 \big]^{1/2}  . \]
  Since $\mathbb{E} \big[ Y_k^2 \, |\, Z = 0  \big]  \le \mathbb{E} \big[ Y_k^2 \big]$, combining the above establishes \eqref{e:ldet3} and hence \eqref{e:ldet1}. Finally, \eqref{e:ldet2} follows from \eqref{e:ldet1} since $\mathbb{E} \big[ X_{i, j}^2 \,| \, Z = 0  \big]  \le \mathbb{E} \big[ X_{i, j}^2 \big]$.
\end{proof}
 
We now proceed to the proofs of Lemmas \ref{l:1}--\ref{l:4}. For this we recall that $f$ is centred, which implies that $\nabla f(x)$ and $\nabla^2 f(x)$ are also centred Gaussian random vectors. 
 
\begin{proof}[Proof of Lemma \ref{l:1}]
By the Cauchy-Schwarz inequality, $I^1(x, y; s, t)$ is bounded above by
\[      \gamma_{x,y}^1(s, t, 0, 0)  \max_{z \in \{x, y\}} \mathbb{E}[ |\textup{det}(\nabla^2 f(z)) |^2   \, | \,  f(x) = s , f(y) = t , \nabla f(x) = \nabla f(y) = 0 ]   .  \]
Applying Lemma \ref{l:det} (more precisely \eqref{e:ldet2}) with the setting $d = n= 2$, this is bounded by
\[      \frac{c}{ \sqrt{ \textup{det}(\Sigma_3(x,y))} }  \Big(    \sup_{z \in D} \max_{|\alpha|=2}   \mathbb{E} \big[ (\partial^\alpha f (z))^2 \big]  \Big)^2  \max\Big\{  1 , \frac{ \sup_{z \in D} \max_{|\alpha|=1}  \big( \mathbb{E} \big[ (\partial^\alpha f(z))^2 \big] \big)^4}{ \textup{det}(\Sigma_2(x,y))^2  }  \Big\}  ,   \]
where $c > 0$ is an absolute constant, and $\Sigma_3(x,y)$ denotes the covariance matrix of the vector
\begin{equation}
\label{d:s3}
( f(x),  f(y), \nabla f(x), \nabla f(y) )   .
\end{equation}
Since, by Gaussian regression, $\textup{det}(\Sigma_2(x,y)) =   \textup{det}(\Sigma_3(x,y)) / \textup{det}(\Sigma_1(x,y)) $, the result follows from \eqref{e:norm2} and~\eqref{e:l1}.
\end{proof}

\begin{proof}[Proof of Lemma \ref{l:2}]
Arguing as in the proof of Lemma \ref{l:1}, and this time applying \eqref{e:ldet1} of Lemma~\ref{l:det} with the setting $d=1$ and $n=2$, there exists a $c > 0$ such that
\[ I^2(x, y; s)  \le    \frac{ c N(x,y)  }  { \sqrt{ \textup{det}(\Sigma_4(x,y)) } }   \max \Bigg\{ 1 ,  \frac{ \max_{|\alpha|=1}  \big( \mathbb{E} \big[ (\partial^\alpha f(x))^2 \big] \big)^4}{  \sigma_1^2(x,y) ^2} \Bigg\}, \]
where 
\begin{equation}
\label{d:n}
 N(x,y) =  \prod_{\text{largest two}}    \mathbb{E} \big[ (\nabla^2 f(x))^2_{i, j} \, | \, \nabla f(x) = \nabla f(y) = 0 \big]  ,
 \end{equation}
 and $\Sigma_4(x,y)$ and $\sigma_1^2(x,y)$ denote respectively the covariance matrix of the vectors
\begin{equation}
\label{d:s4}
  ( f(x), \nabla f(x), \nabla f(y) ) \quad \text{and} \quad f(x) \, | \, (\nabla f(x), \nabla f(y) ) .
  \end{equation}

Given \eqref{e:norm2}, it remains to examine the asymptotics, as $|x-y| \to 0$, of the quantities $N(x,y)$,  $\textup{det}(\Sigma_4(x,y))$ and $\sigma_1^2(x,y)$. In particular it is sufficient to prove that, as $|x-y| \to 0$,
\begin{enumerate}
\item  $ N(x,y) = O(|x-y|^2)$;
\item There exists a $c_3 > 0$ such that $ \textup{det}(\Sigma_4(x,y) ) > c_3 |x-y|^4 + O(|x-y|^6)$;  
\item There exists a $c_4 > 0$ such that $\sigma_1^2(x,y)  > c_4 + O(|x-y|^2)$;
\end{enumerate}
where $c_3, c_4$ and the constants implicit in $O(\cdot)$ depend only the constants $c_1$ and $c_2$ defined in \eqref{e:l21} and \eqref{e:l22}.

Let us finish the proof by validating the claimed asymptotics. For this, we rely on the following matrix computation (whose proof is simple to verify):

\begin{lemma}
\label{l:matrix}
For parameters $a_1, a_2, a_3, b_1, b_2 \in \mathbb{R}$, define the matrices  
\[   A_1 =  \left[ 
\begin{array}{ccccc}
1 &  0 & 0 & -a_1 & 0 \\
0 & 1 & 0 & 1-a_2 & 0 \\
0 & 0 & 1 & 0 & 1-a_3 \\
-a_1 & 1-a_2 & 0 & 1 & 0 \\
0 & 0 & 1-a_3 & 0 & 1 \\
\end{array}
\right]    , \ 
A_2 =  \left[ 
\begin{array}{cccc}
 1 & 0 & 1-a_2 & 0 \\
 0 & 1 & 0 & 1-a_3 \\
  1-a_2 & 0 & 1 & 0 \\
  0 & 1-a_3 & 0 & 1 \\
\end{array} 
   \right]  \]
   and
   \[  A_3 =  \left[ 
\begin{array}{ccc}
   0 & 0 & 0   \\
      0 & 0 & 0  \\
         b_1 & 0 & b_2  \\
            0 & b_2 & 0   \\
  \end{array} 
\right].  \]
 Then 
  \[   \textup{det}(A_1) = ( 2a_2 - a_1^2 - a_2^2 ) ( 2a_3 - a_3^2  )  \quad \text{and} \quad   \textup{det}(A_2) = a_2 a_3 (2-a_2)(2-a_3)  . \]
  Moreover, assuming that $\textup{det}(A_2) \neq 0$, the diagonal elements of $A_3^T (A_2)^{-1} A_3$ are equal to $ \Big(  \frac{b_1^2}{2a_2 - a_2^2}, \frac{b_2^2}{2a_3 - a_3^2}, \frac{b_2^2}{2a_2 - a_2^2} \Big) $.
\end{lemma}
  
Fix $x = (x_1, x_2)  \in D$ and define $K_x(y) = \mathbb{E}[f(x) f(y)]$. Under the normalisation \eqref{e:norm2}, and since $K_x$ is $C^3$, we may write (implicitly evaluating derivatives of $K_x$ at $x$), 
 \begin{align*}
 & K_x(y_1, y_2)  = 1 - \frac{(y_1-x_1)^2}{2}   - \frac{(y_2-x_2)^2}{2}   \\
 &  \qquad + \frac{\partial^{(4, 0)}K_x }{24} \, (y_1-x_1)^4 + \frac{\partial^{(2, 2)}K_x }{6} \, (y_1-x_1)^2 (y_2-x_1)^2  + \frac{\partial^{(0, 4)}K_x }{24}  \, (y_2-x_2)^4 + O(|x-y|^6) , 
\end{align*}
where the constant implicit in $O(\cdot)$ depends only on $c_1$ (defined in \eqref{e:l21}). Let us suppose, without loss of generality, that $y =x + (r, 0)$, for $r > 0$. Recall $\Sigma_2(x,y)$ defined in \eqref{d:s12}, and denote by $\Sigma_5(x,y)$ the covariance matrix between
  \[ ( \nabla f(x), \nabla f(y) )   \quad \text{and}   \quad  \nabla^2 f(x)  , \]
  and by $\Sigma_6(x)$ the covariance matrix of $\nabla^2 f(x)$, considered as the vector
  \[   ( \partial^{(2,0)} f(x),   \partial^{(1,1)} f(x), \partial^{(0,2)} f(x)) . \]  
  Computing the entries explicitly, observe that $\Sigma_4(x,y)$, $\Sigma_2(x,y)$ and $\Sigma_5(x,y)$ have the structure of the matrices $A_1, A_2$ and $A_3$ respectively in Lemma \ref{l:matrix}, with parameter settings
\[  a_1  = r + O(r^3)   \, , \quad a_2  = \frac{\partial^{(4, 0)}K_x }{2} \, r^2 + O(r^4)   \, , \quad a_3 = \frac{\partial^{(2, 2)}K_x }{2} \, r^2 + O(r^4)   ,\]
  \[ b_1 = \partial^{(4, 0)}K_x \, r + O(r^3)   \quad \text{and} \quad b_2    = \partial^{(2, 2)}K_x \, r + O(r^3). \]
 Applying Lemma \ref{l:matrix}, 
\begin{align*}
\textup{det}(\Sigma_4(x,y)) &= \big( 2a_2 - a_1^2 - a_2^2 \big) \big( 2a_3 - a_3^2  \big)  \\
& = \big(  \partial^{(4, 0)} K_x \, r^2 - r^2 + O(r^4) \big) \big(   \partial^{(2, 2)} K_x \, r^2 + O(r^4) \big) \\
& =  (\partial^{(4, 0)} K_x - 1) \, \partial^{(2, 2)} K_x \, r^4 + O(r^6) .
\end{align*}
Since $  \partial^{(4, 0)} K_x =    \mathbb{E}[ (\partial^{(2, 0)} f(x))^2]  $ and $   \partial^{(2, 2)} K_x  = \mathbb{E}[ (\partial^{(1, 1)} f(x))^2] $, the claimed asymptotics for $\textup{det}(\Sigma_4(x,y))$ follow from \eqref{e:l22}. Again applying Lemma~\ref{l:matrix}, the diagonal elements of $  \Sigma_5(x,y)^T \Sigma_2(x,y)^{-1} \Sigma_5(x,y) $ are equal, respectively, to 
\[ \frac{b_1^2}{2a_2 - a_2^2}  =   \frac{  (\partial^{(4, 0)} K_x )^2  \, r^2 + O(r^4) }{ \partial^{(4, 0)} K_x  \, r^2 + O(r^4) } =   \partial^{(4, 0)} \kappa  + O(r^2) , \]
\[  \frac{b_2^2}{2 a_3 - a_3^2 }  = \frac{  (\partial^{(2, 2)} K_x )^2  \, r^2 + O(r^4) }{ \partial^{(2, 2)} K_x  \, r^2 + O(r^4) } =   \partial^{(2, 2)} K_x  + O(r^2)    \]
and
\[   \frac{b_2^2}{2a_2 - a_2^2} =  \frac{  (\partial^{(2, 2)} K_x )^2  \, r^2 + O(r^4) }{ \partial^{(4, 0)} K_x  \, r^2 + O(r^4) } =   \frac{ (\partial^{(2, 2)} K_x)^2}{ \partial^{(4, 0)} K_x} + O(r^2)  . \]
On the other hand, by explicit computation the diagonal elements of $\Sigma_6(x)$ are equal to $\big( \partial^{(4, 0)} K_x  ,  \,  \partial^{(2, 2)} K_x  ,  \,  \partial^{(0, 4)} K_x     \big) $, and so the diagonal elements of $\Sigma_6(x) - \Sigma_5(x,y)^T \Sigma_2(x,y)^{-1} \Sigma_5(x,y) $ are equal to $( 0, 0, O(1)  )  + O(r^2) $. Since by Gaussian regression these diagonal elements are 
\[    \mathbb{E}[ (\nabla^2 f(x))^2_{i,j}   \, | \, ( \nabla f(x) = \nabla f(y) = 0) ] \]
for $(i,j)= (1,1), (1,2), (2,2)$ respectively, we deduce that $N(r) = O(r^2)$ as claimed. Finally, by Gaussian regression, 
\[  \sigma_1^2(x,y)  = \text{det}(\Sigma_4(x,y)) / \text{det}(\Sigma_2(x,y)) ,  \]
and since, by Lemma \ref{l:matrix},
\[  \text{det}(\Sigma_2(x,y))  = a_2 a_3 (2-a_2)(2-a_3)  =  \partial^{(4, 0)} K_x \,  \partial^{(2, 2)} K_x \, r^4 + O(r^6)   \]
we have that
\[  \sigma_1^2(x,y)  =   (\partial^{(4, 0)} K_x - 1)  /   \partial^{(4, 0)} K_x  + O(r^2) > c_4 + O(r^2),\]
as claimed.
\end{proof}

\begin{proof}[Proof of Lemma \ref{l:3}]
Arguing as in the proof of Lemma \ref{l:1}, and this time applying Lemma~\ref{l:det} (more precisely \eqref{e:ldet2}) with the setting $d=1$ and $n=1$, there exists a constant $c > 0$ such that
\[ I^3(x;s)  \le  \frac{c}  { \sqrt{ \textup{det}(\Sigma_7(x)) } } \, \max_{|\alpha|=2}   \mathbb{E} \big[ (\partial^\alpha f(x))^2 \big]    \, \max \Big\{ 1 \, , \, \frac{\max_{|\alpha|=1}  \big( \mathbb{E} \big[ (\partial^\alpha f(x))^2 \big] \big)^2}{\sigma_2(x)^2} \Big\} , \]
where $\Sigma_7(x)$ denotes the covariance matrix of $( f(x), \nabla f(x)  )$, and $\sigma_2^2(x)$ denotes the variance of $f(x) \, | \, \nabla f(x)$. Under the normalisation \eqref{e:norm2}, $\Sigma_7(x) =  \sigma_2^2(x) = 1$, and the result follows from~\eqref{e:l3}.
\end{proof}

\begin{proof}[Proof of Lemma \ref{l:4}]
Arguing as in the proof of Lemma \ref{l:1}, and applying H\"{o}lder's inequality as in the proof of Lemma \ref{l:det}, there exists a $c > 0$ such that
\[ I^4(x, y)  \le    \frac{ c N(x,y)  }  { \sqrt{ \textup{det}(\Sigma_2(x,y)) } }  \le     \frac{ c  \, \Big( \max_{z \in \{x, y\}} \max_{|\alpha|=2}   \mathbb{E} \big[ (\partial^\alpha f (z))^2 \big]  \Big)^2   }  { \sqrt{ \textup{det}(\Sigma_2(x,y)) } }   , \]
where $N(x, y)$ is defined in \eqref{d:n}. Moreover, by Gaussian regression and the normalisation \eqref{e:norm2},
\[  \textup{det}(\Sigma_2(x,y))   = \frac{ \textup{det}(\Sigma_3(x,y)) }   {  \text{det} (\Sigma_1(x, y) )} \ge  \frac{ \textup{det}(\Sigma_3(x,y))  } {  \text{Cov}[ f(x), f(y) ]}  \ge  \textup{det}(\Sigma_3(x,y))  \]
and similarly
\[  \textup{det}(\Sigma_2(x,y))   = \frac{ \textup{det}(\Sigma_4(x,y))  } {\sigma_1^1(x) } \ge  \frac{ \textup{det}(\Sigma_4(x,y))  } {  \textrm{Var}[ f(x) ] } =  \textup{det}(\Sigma_4(x,y)) ,  \]
where $\Sigma_1, \Sigma_3, \Sigma_4$ and $\sigma_1^2$ are defined in \eqref{d:s12}, \eqref{d:s3} and \eqref{d:s4}. Hence
\[ I^4(x, y)  \le    \frac{ c N(x,y)  }  { \sqrt{ \textup{det}(\Sigma_4(x,y)) } }  \le     \frac{ c  \Big( \sup_{z \in \{x,y\}} \max_{|\alpha|=2}   \mathbb{E} \big[ (\partial^\alpha f (z))^2 \big]  \Big)^2   }  { \sqrt{ \textup{det}(\Sigma_3(x,y)) } }   , \]
and the uniform bound on $I^4$ follows as in the proofs of Lemmas \ref{l:1} and \ref{l:2}. Similarly
\[ I^5(x)  \le     \frac{ c  \max_{|\alpha|=2}   \mathbb{E} \big[ (\partial^\alpha f (x))^2 \big]   }{ \sqrt{ \text{det}( \text{Cov}[\nabla f(x), \nabla f(x) ] ) } } , \]
which is uniformly bounded by \eqref{e:norm2} and \eqref{e:l41}.
\end{proof}


\section{Appendix: The one-dimensional case}
\label{s:a}

Analogous bounds also hold in the one-dimensional case. Let $f$ be a $C^1$-smooth stationary Gaussian process, with $\kappa(x) = \textrm{Cov}[f(0),f(x)]$ its covariance kernel. The analogue of Condition \ref{c:1} is the following:
 
 \begin{condition}
\label{c:11d}
\
\begin{itemize}
\item The covariance kernel $\kappa$ is of class $C^6$.
\item For each $x \in \mathbb{R} \setminus \{0\}$, the Gaussian vector $( f(0), f(x),  f'(0),  f'(x) )$ is non-degenerate.
\item As $|x| \to \infty$, $\max_{\alpha \le 2} |  \partial^\alpha \kappa(x) | \to 0$.
\end{itemize}
\end{condition}

For each $R > 0$ and $a \le b$, let $N_R[a, b]$ denote the number of critical points of $f$ in the interval $[0,R]$ whose heights (i.e.\ `critical values') lie in the interval $[a, b]$, i.e.,
\[ N_R[a, b] = \# \{  x \in [-R, R] :  f(x) \in [a, b],  f'(x) = 0 \}  .\]
Then we have the following bound on the second moment of $N_R[a,b]$:

\begin{theorem}
\label{t:smb1d}
Suppose $f$ satisfies Condition \ref{c:11d}. Then there exists a $c > 0$ such that, for all $R \ge 1$ and $a \le b$, 
\[  \mathbb{E}[N_R[a, b]^2]  \le c \,   \min\{ R^2 (b-a)^2 + R (b-a) , \, R^2  \}    . \]
\end{theorem}

\begin{remark}
In the one-dimensional case we can omit the extra condition, analogous to Condition \ref{c:2}, that the spectral measure of $f$ is not supported on two points, since this is already implied by Condition \ref{c:11d}.
\end{remark}

We can also state a uniform bound analogous to Theorem \ref{t:smb2}.

\begin{theorem}
\label{t:smb21d}
 Let $(f_i)_{i \in \mathcal{I}}$ be a collection of continuous (not necessarily stationary) Gaussian processes, each defined on a compact interval $D_i \subset \mathbb{R}$, with a $C^{3,3}$-smooth covariance kernel. Let $N_i[a, b]$ be the number of critical points of $f_i$ on $D_i$ whose heights lie in $[a, b]$. Suppose that:
\begin{enumerate}
\item The processes are normalised so that, for each $i \in \mathcal{I}$ and $x \in D_i$,
\[ \mathbb{E}[f_i(x)]  = 0 \, , \ \textrm{Var}[f_i(x)] = 1 \, , \ \text{Cov}[f_i(x), f'_i(x) ] = 0 \quad \text{and} \quad   \textrm{Var}[ f'_i(x)] = 1, \]
and $( f_i(x), f_i(y),  f'_i(x),  f'_i(y) )$ is non-degenerate for all distinct $x, y \in D_i$;
\item There exists a constant $c_1 > 0$ such that $\sup_{i \in \mathcal{I} } \sup_{x \in D_i } \max_{|\alpha| \le 3}   \textrm{Var}[\partial^{\alpha} f_i(x)] < c_1$;
\item There exists a constant $c_2>0$ such that $ \inf_{i \in \mathcal{I} } \inf_{x \in D_i}   \textrm{Var}[ f''_i(x)^2] > 1 + c_2 $;
\item For each $\delta> 0$, there exists a constant $c_3 > 0$ such that
\[ \inf_{i \in \mathcal{I} } \inf_{|x-y|  \ge \delta}  \det(\Sigma^i_1(x,y))  > c_3 \quad \text{and} \quad \inf_{i \in \mathcal{I} } \inf_{|x-y|  \ge \delta}  \det(\Sigma^i_2(x,y))  > c_3, \]
where  $\Sigma^i_1(x,y)$ and $\Sigma^i_2(x)$ denote the covariance matrices of $(f_i(x), f_i(y) ) \, | \, (f_i'(x), f_i'(y) )$  and $(f_i'(x), f_i'(y) )$ .
\end{enumerate}
Then there exists a $c > 0$ such that, for all $i \in \mathcal{I}$ and $a \le b$,
\[  \mathbb{E}[N_i[a, b]^2]  \le c \,   \min\{ \text{Len}(D_i)^2 (b-a)^2 + \text{Len}(D_i) (b-a) , \,  \text{Len}(D_i)^2 + \text{Len}(D_i) \}   . \]
\end{theorem}

The proof of Theorems \ref{t:smb1d} and \ref{t:smb21d} are identical to the proofs of Theorems \ref{t:smb} and \ref{t:smb2}, save for the obvious changes in notation. Indeed, in this case we only require simplified versions of the auxiliary Lemmas \ref{l:det} and \ref{l:matrix}. 

\medskip

\bibliographystyle{plain}
\bibliography{biblio}

\begin{thebibliography}{10}

\bibitem{AL}
J.-M. Aza\"{i}s and J.~Le{\'{o}}n.
\newblock Necessary and sufficient conditions for the finiteness of the second
  moment of the measure of level sets.
\newblock {\em arXiv preprint, arxiv:1905.12342}, 2019.

\bibitem{azais2009level}
J.-M. Aza{\"i}s and M.~Wschebor.
\newblock {\em Level sets and extrema of random processes and fields}.
\newblock Wiley, 2009.

\bibitem{beliaev2017two}
D.~Beliaev, V.~Cammarota, and I.~Wigman.
\newblock Two point function for critical points of a random plane wave.
\newblock {\em Int. Math. Res. Notices}, 2017:1--29, 2017.

\bibitem{bmm19}
D.~Beliaev, M.~McAuley, and S.~Muirhead.
\newblock Fluctuations of the number of excursion sets of planar {G}aussian
  fields.
\newblock {\em arXiv preprint, arxiv:1908.10708}, 2019.

\bibitem{topmixing18}
D.~Beliaev, S.~Muirhead, and A.~Rivera.
\newblock A covariance formula for topological events of smooth {G}aussian
  fields.
\newblock {\em arXiv preprint, arxiv:1811.08169}, 2018.

\bibitem{Berry}
M.V. Berry.
\newblock Regular and irregular semiclassical wavefunctions.
\newblock {\em J. Phys. A.}, 10(12):2083--2091, 1977.

\bibitem{cammarota2014distribution}
V.~Cammarota, D.~Marinucci, and I.~Wigman.
\newblock On the distribution of the critical values of random spherical
  harmonics.
\newblock {\em J. Geom. Anal.}, 26(4):1--73, 2014.

\bibitem{cammarota2017fluctuations}
V.~Cammarota and I.~Wigman.
\newblock Fluctuations of the total number of critical points of random
  spherical harmonics.
\newblock {\em Stoc. Proc. Appl.}, 127(12):3825--2869, 2017.

\bibitem{cheng2015expected}
D.~Cheng and A.~Schwartzman.
\newblock Expected number and height distribution of critical points of smooth
  isotropic gaussian random fields.
\newblock {\em Bernoulli}, 24(4B):3422--3446, 2018.

\bibitem{eli}
A.I. Elizarov.
\newblock On the variance of the number of stationary points of a homogeneous
  {G}aussian field.
\newblock {\em Theory Probab. Appl.}, 29(3):569--570, 1985.

\bibitem{ef_2016}
A.~Estrade and J.~Fournier.
\newblock Number of critical points of a {G}aussian random field: {C}ondition
  for a finite variance.
\newblock {\em Statist. Probab. Lett.}, 118:94--99, 2016.

\bibitem{EL}
A.~Estrade and J.~Le{\'{o}}n.
\newblock A central limit theorem for the {E}uler characteristic of a
  {G}aussian excursion set.
\newblock {\em Ann. Probab.}, 44(6):3849--3878, 2016.

\bibitem{geman}
D.~Geman.
\newblock On the variance of the number of zeros of a stationary {G}aussian
  process.
\newblock {\em Ann. Math. Statist.}, 43(3):977--982, 1972.

\bibitem{KL}
M.~Kratz and J.R. Le{\'{o}}n.
\newblock On the second moment of the number of crossings by a stationary
  {G}aussian process.
\newblock {\em Ann. Probab.}, 34(4):1601--1607, 2006.

\bibitem{nic_2017}
L.I. Nicolaescu.
\newblock A {CLT} concerning critical points of random functions on a
  {E}uclidean space.
\newblock {\em Stoc. Proc. Appl.}, 127(10):3412--3446, 2017.

\bibitem{NPR}
I.~Nourdin, G.~Peccati, and M.~Rossi.
\newblock Nodal statistics of planar random waves.
\newblock {\em Commun. Math. Phys.}, 369(1):99--151, 2019.

\end{thebibliography}

\end{document}